\documentclass[a4paper,fleqn]{cas-dc}

\usepackage[numbers]{natbib}
\usepackage{amsmath,amssymb}
\usepackage[ruled,linesnumbered]{algorithm2e}

\def\tsc#1{\csdef{#1}{\textsc{\lowercase{#1}}\xspace}}
\tsc{WGM}
\tsc{QE}
\newtheorem{theorem}{Theorem}
\newtheorem{lemma}[theorem]{Lemma}
\newtheorem{corollary}[theorem]{Corollary}
\newtheorem{proposition}[theorem]{Proposition}
\newdefinition{remark}{Remark}
\newtheorem{assumption}{Assumption}
\newproof{proof}{Proof}
\newcommand{\bR}{\mathbb{R}}

\newcommand{\cI}{\mathcal{I}}

\newcommand{\cQ}{\mathcal{Q}}
\newcommand{\cN}{\mathcal{N}}

\newcommand{\bx}{\boldsymbol{x}}
\newcommand{\by}{\boldsymbol{y}}
\newcommand{\bz}{\boldsymbol{z}}

\newcommand{\bu}{\boldsymbol{u}}

\newcommand{\bphi}{\boldsymbol{\phi}}

\newcommand{\bg}{\boldsymbol{g}}

\newcommand{\bq}{\boldsymbol{q}}
\newcommand{\bv}{\boldsymbol{v}}

\newcommand{\be}{\begin{equation}}
\newcommand{\ee}{\end{equation}}
\newcommand{\iprod}[2]{\langle #1, #2 \rangle}

\newcommand{\dist}{\mathrm{dist}}

\newcommand{\gph}{\mathrm{gph}}

\newcommand{\Range}{\mathrm{Range}}
\newcommand{\Null}{\mathrm{Null}}
\DeclareMathOperator*{\argmin}{arg\,min}

\begin{document}
\let\WriteBookmarks\relax
\def\floatpagepagefraction{1}
\def\textpagefraction{.001}

\shorttitle{Linear convergence analysis of Proximal-Tracking}    

\shortauthors{T. Yuan and X. Xiao}  

\title [mode = title]{Linear convergence rate analysis of Proximal-Tracking for nonsmooth distributed optimization}  

%\tnotemark[1]{This} 

%\tnotetext[1]{This} 


\author[1]{Tianyu Yuan}%[<options>]


\ead{y22301078@mail.dlut.edu.cn}


\credit{Writing - Original draft preparation, Formal analysis, Methodology}

\affiliation[1]{organization={School of Mathematical Sciences, Dalian University of Technology},
            %addressline={}, 
            city={Dalian},
%          citysep={}, % Uncomment if no comma needed between city and postcode
            %postcode={}, 
            %state={},
            country={China}}

\author[1]{Xiantao Xiao}\cormark[1]%[]

% Footnote of the second author

% Email id of the second author
\ead{xtxiao@dlut.edu.cn}


\credit{Writing - review \& editing, Supervision, Methodology, Funding acquisition, Conceptualization}

% \affiliation[2]{organization={},
% %          citysep={}, % Uncomment if no comma needed between city and postcode

\cortext[cor1]{Corresponding author}



\begin{abstract}
Proximal-Tracking is a novel algorithm proposed in \cite{FP2022} for nonsmooth distributed consensus optimization with local set constraints. The global convergence and numerical behavior of Proximal-Tracking have been well studied in \cite{FP2022}. As a complement, in this paper we provide a theoretical analysis regarding its linear convergence rate under additional strong convexity and calmness assumptions. In particular, this analysis does not require assumptions such as differentiability or smoothness, thus preserving the original advantages of the algorithm.
\end{abstract}


% % Research highlights
% \begin{highlights}
% \item Proximal-Tracking is reformulated as an equivalent EXTRA-type form.
% \item Linear convergence rate of Proximal-Tracking is established under strong convexity and calmness, without smoothness.
% \item The results extends to Augmented Lagrangian Tracking for distributed optimization with coupling constraints.
% \end{highlights}

%\nocite{*}

% Keywords

\maketitle

%------------------------------------------------------------------------
%------------------------------------------------------------------------
%------------------------------------------------------------------------
\section{Introduction}\label{sec:intro}

Distributed optimization plays a central role in a wide range of applications, including large-scale learning and networked systems. In such settings, multiple agents cooperate to solve a global optimization problem using only local information and limited communication with their neighbors. A substantial literature has developed distributed algorithms for unconstrained problems, and these methods have gradually been extended to more general and challenging settings. For comprehensive surveys of recent developments, we refer interested readers to \cite{review2019} and \cite{review2022}.

In this paper, we focus primarily on distributed consensus optimization problems with simple set constraints. In such scenarios, each agent is associated with a local objective and subject to local set constraints, with the goal of minimizing the aggregated objective function while satisfying all set constraints and consensus conditions.
A significant body of research has been dedicated to the treatment of such problems. Initially, for cases with consistent local set constraints, Nedi\'{c} et al. \cite{5404774} proposed a distributed projected (sub)gradient algorithm with decreasing step size and proved its convergence on fixed fully connected graphs. 
%This approach was subsequently extended to time-varying directed graphs (e.g. \cite{LIN2016120}, \cite{WLRS2018}). 
To accelerate convergence, a distributed primal-dual algorithm with fixed-steps was proposed  in \cite{LEI2016110}  based on EXTRA \cite{EXTRA}. Building upon the classical alternating direction method of multipliers (ADMM),  an algorithm named D-ADMM was introduced in \cite{D-ADMM} to address this problem. 

Among the many algorithms, the Proximal-Tracking method proposed in \cite{FP2022} is particularly noteworthy. It cleverly combines the classic proximal minimization algorithm \cite[Section 5.1]{Bertsekas2015} with the gradient tracking framework \cite{DIGing} to solve this distributed optimization problem. The main advantage of this algorithm is that it is a fixed-step method specifically designed for problems with explicit local set constraints, requiring only convexity of the local cost function and not assuming differentiability, smoothness, or Lipschitz continuity. Moreover, Proximal-Tracking can also be applied to solve distributed optimization with equality and inequality coupling constraints, as demonstrated in \cite{FP2023}.

Despite these advantages, the theoretical understanding of Proximal-Tracking remains incomplete. In particular, no precise analysis of the convergence rate has been provided. This limitation hinders our deeper understanding of the algorithm and necessitates a more detailed theoretical analysis. 

In this paper, we present an equivalent formulation of Proximal-Tracking, linking it to the framework of EXTRA \cite{EXTRA}, thereby broadening our perspective on the algorithm. Based on this formulation, we establish the linear convergence rate of Proximal-Tracking. This theoretical result requires the additional assumptions of strong convexity and calmness, while maintaining the key advantage of the algorithm of not requiring assumptions such as differentiability or smoothness of the objective function.

The remainder of this paper is organized as follows. In Section \ref{sec:prel}, we start by introducing the problem formulation, then present the Proximal-Tracking algorithm and its equivalent formulation. In Section \ref{sec:rate}, we study the linear convergence of Proximal-Tracking under the assumptions of strong convexity and calmness. Section \ref{sec:alt} provides a short discussion on the linear convergence of Augmented Lagrangian Tracking.

\textbf{Notation.} Let $\mathbb{R}^n$ denote the $n$-dimensional Euclidean space. For a symmetric matrix $A$, let $\rho(A)$ denote its spectral radius. The symbol $\otimes$ represents the Kronecker product. For a set $Z$, $\cI_{Z}(z)$ denotes the indicator function of the set $Z$, which is defined as $\cI_{Z}(z) = 0$ if $z \in Z$ and $\cI_{Z}(z) = +\infty$ otherwise. Let $\dist(z,Z):=\min_{x\in Z}\|x-z\|$ be the distance from a point  $z$ to set $Z$.   For a convex function $f$, $\partial f(z)$ denotes the subdifferential of $f$ at $z$. The vector in $\bR^N$ containing all ones is denoted by $\mathbb{1}_{N}$. Let  $I$ be the identity matrix of size  $Nn$. For a set-valued mapping $F:\bR^n\rightrightarrows\bR^m$,  the graph of $F$ is $\gph(F):=\{(x,y)|y\in F(x)\}$ and its inverse mapping is defined by $F^{-1}(y):=\{x\in\bR^n|y\in F(x)\}$.

%------------------------------------------------------------------------
%------------------------------------------------------------------------
%------------------------------------------------------------------------
\section{Preliminaries} \label{sec:prel}
In this section, let us recall the algorithmic framework of Proximal-Tracking and the existing convergence results.

\subsection{Problem formulation and assumptions}
Consider a multi-agent network consisting of $N$ agents, which cooperatively solve the following problem:
\be \label{eq:IPMO}
    \begin{aligned}
        \min_{z} &\; f(z):=\sum_{i=1}^{N} f_i(z) ,\\
        \text{s.t. }& z \in \bigcap_{i=1}^{N} Z_i.
    \end{aligned}
\ee
where each $f_i:\bR^n\rightarrow\bR$ is a local objective function and can only be known by the agent $i$. Meanwhile, the consensus decision variable $z$ must satisfy the local constraint set $Z_i \subseteq \bR^n$ of each agent $i$.

The following assumption is made to ensure that the above distributed problem is well defined. 
\begin{assumption}\label{assu:monotone}
For any $i=1,\ldots,N$, the local objective function $f_i$ is convex and lower semicontinuous,
and the local set $Z_i$ is convex and closed.  Moreover,  
there exists at least one optimal solution $z^*$ to Problem \eqref{eq:IPMO}.
\end{assumption}

For each agent $i$, define locally extended real-valued function $\phi_i:\bR^n\rightarrow(-\infty,+\infty]$ by
 $\phi_i(z):= f_i(z)+ \cI_{Z_i}(z)$. 
 Problem \eqref{eq:IPMO} can be equivalently reformulated as: 
\be\label{eq:IPMO_eqv}
\begin{aligned}
    \min_z \phi(z):= \sum_{i=1}^{N} \phi_i (z).
\end{aligned}
\ee
Assumption \ref{assu:monotone} and \cite[Theorem 9.3]{ConvexAnalysis} ensure that all $\phi_i$ and $\phi$ are proper, lower semicontinuous, and convex functions.

In order to solve \eqref{eq:IPMO} cooperatively, agents must exchange information repeatedly through a communication network. In this process, communication between agents is modeled as a graph $ \mathcal{G} = (\mathcal{V}, \mathcal{E})$, where the set of vertices $\mathcal{V} = \{1 \ldots N \} $ denotes the agents, and the set of edges $\mathcal{E} \subseteq \mathcal{V} \times \mathcal{V} $ denotes the communication links between the agents. We denote the set of neighbors of agent $i$ by $\mathcal{N}_{i} = \{ j \in \mathcal{V} : (i,j) \in \mathcal{E} \}$. Define the adjacency matrix $\mathcal{W} = (w_{ij}) \in \mathbb{R}^{N \times N} $ with $w_{ij} > 0$ if $(i,j) \in \mathcal{E}$ and $w_{ij} = 0$ otherwise.
The following assumption is popular in the study of distributed optimization.

\begin{assumption}\label{assu:connectivity}
    The graph $\mathcal{G}$ is undirected and connected. The matrix $\mathcal{W}$ is symmetric, doubly stochastic, and positive definite. 
\end{assumption}

\begin{remark}
The assumption that $\mathcal{W}$ is positive definite can be satisfied by constructing $\mathcal{W}=\frac{3I_N+\mathcal{W}'}{4}$ from any adjacency matrix $\mathcal{W}'$ that is symmetric and doubly stochastic. Similar approach is also used in \cite{EXTRA,PG-EXTRA,FP2022}.
\end{remark}

Let
\[
\sigma:=1-\max\{\lambda:\lambda\in\mbox{eigen}(\mathcal{W}),\lambda<1\}
\]
be the spectral gap of $\mathcal{W}$. That is, $\sigma=1-\lambda_2(\mathcal{W})$ when the eigenvalues are arranged in nonincreasing order. 
Define $W := \mathcal{W} \otimes I_{n}$ and
\[
W_{\infty} := \left(\frac{1}{N} \mathbb{1}_{N} \mathbb{1}_{N}^\top \right) \otimes I_{n},\quad \tilde{W} := W - W_{\infty}.
\]
Under Assumption \ref{assu:connectivity}, for all $\theta \in \bR^{n} $ and $\beta_{i} \in \bR^{n}, i = 1 \cdots N$,  we have from  \cite[Lemma 1]{FP2022} that
\begin{subequations} \label{eq:PropW}
    \begin{align}
        & W_{\infty} \boldsymbol{\beta} = \boldsymbol{\bar{\beta}}, \label{eq:PropW1} \\
        & W_{\infty} W = W W_{\infty} = W_{\infty}, \label{eq:PropW2} \\
        & W ( \mathbb{1}_{N} \otimes \theta ) = \mathbb{1}_{N} \otimes \theta, \label{eq:PropW3} \\
        & W_{\infty} ( \mathbb{1}_{N} \otimes \theta ) = \mathbb{1}_{N} \otimes \theta, \label{eq:PropW4} \\
        & W_{\infty} (\boldsymbol{\beta} - \boldsymbol{\bar{\beta}}) = 0, \label{eq:PropW5}
    \end{align}
where
 $ \boldsymbol{\beta} := {[{\beta_{1}}^\top \cdots {\beta_{N}}^\top ]}^\top $ and $\boldsymbol{\bar{\beta}} := \mathbb{1}_{N} \otimes ( \frac{1}{N} \sum_{i=1}^{N} \beta_{i})$. Moreover, it yields
    \begin{align}
        & \rho (\tilde{W}) < 1, \label{eq:PropW6} \\
        & (I - \tilde{W})^{-1} (I-W) = I - W_{\infty}. \label{eq:PropW7}     
    \end{align}
\end{subequations}
It is also easy to verify that
\[
W\succ 0,\quad 2I-W\succ 0,
\]
the eigenvalues of $W$ lie in $(0,1]$ and $\lambda_{\max}(W)=1$. For every $\bv\in\Range(I-W)$, $\|(I-W)\bv\|\geq\sigma \|\bv\|$.

\begin{lemma}\label{lem:span}
Under Assumptions \ref{assu:monotone} and \ref{assu:connectivity}, it holds that
    \[{\Range}(I-W) = {\Null}(W_{\infty}).\]
\end{lemma}
\begin{proof}
It is obvious that ${\Range}(I-W)\subseteq \Null(W_{\infty})$ from (\ref{eq:PropW2}). It remains to show that the two subspaces have the same dimension.

 Recall that $W_{\infty}=\big(\frac{1}{N}\mathbb{1}_{N}\mathbb{1}_{N}^\top\big)\otimes I_{n}$, we have that $W_{\infty}$ has rank $n$ and hence the dimension of ${\Null}(W_{\infty})$ is $Nn-n$. Moreover, it is not difficult to derive that  $W_{\infty}$ is idempotent and hence the dimension of ${\Range}(I-W_{\infty})$ is also $Nn-n$.
 Finally, from (\ref{eq:PropW7}), we have that ${\Range}(I-W_{\infty})={\Range}(I-W)$ and thus the proof is completed. \hfill $\square$
\end{proof}

%%---------------------------------------------------
\subsection{Proximal-Tracking}
Proximal-Tracking, proposed in \cite{FP2022}, can be roughly seen as a combination of the classic proximal minimization algorithm  and the popular gradient tracking framework. 

For minimizing a function $f(z)$, the proximal minimization algorithm iteratively updates the decision variable $z$ by solving a sequence of subproblems of the form
\[
z^{k+1}=\argmin\limits_{z\in\bR^n}\left\{f(z)+\frac{1}{2c_k}\|z-z^k\|^2\right\},
\]
where $c_k>0$ is a step size parameter. Obviously, if $f$ is differentiable, the above update can be rewritten as
\be\label{eq:PPA}
z^{k+1}=z^k-c_k\nabla f(z^{k+1}),
\ee
which can be interpreted as an implicit gradient method.

The gradient-tracking technique is a popular approach to improve the convergence behavior of decentralized optimization algorithms \cite{DIGing, QL2018, Tracking-ADMM, Push-Pull, PN2021}.
Consider Problem (\ref{eq:IPMO_eqv}) and let $q_i^k\in\partial \phi_i(z_i^k)$, the (sub)gradient-tracking method can be stated as follows,
\be\label{eq:DIGing}
z_i^{k+1}=\hat{z}_i^k-c g_i^k
\ee
with
\[
\hat{z}_i^k=\sum_{j\in\cN_i}w_{ij}z_j^k,\quad g_i^k=\sum_{j\in\cN_i}w_{ij}g_j^{k-1}+q_i^k-q_i^{k-1},
\]
where  $c>0$ is the stepsize. Clearly, the method (\ref{eq:DIGing}) can be viewed as a kind of (sub)gradient step for agent $i$, in which $\hat{z}_i^k$ is the weighted average of the iterates of its neighboring agents and $g_i^k$ tracks the gradient information according to a dynamic average consensus scheme.

Inspired by  (\ref{eq:PPA}), by simply replacing the (sub)gradient term in (\ref{eq:DIGing}) with the implicit gradient term, we derive the Proximal-Tracking algorithm  as follows
\be\label{eq:PT}
z_i^{k+1}=\hat{z}_i^k-c g_i^{k+1}
\ee
with
\[
\hat{z}_i^k=\sum_{j\in\cN_i}w_{ij}z_j^k,\quad g_i^{k+1}=\sum_{j\in\cN_i}w_{ij}g_j^{k}+q_i^{k+1}-q_i^k.
\]
Note that (\ref{eq:PT}) is not directly implementable since $q_i^{k+1}\in\partial \phi_i(z_i^{k+1})$ is not available at the current iteration. However, it is shown in \cite{FP2022} that the Proximal-Tracking algorithm can be reformulated to a practical implementable form, which is described in Algorithm \ref{alg:PT}.
\begin{algorithm}
    \caption{Proximal-Tracking run by agent $i$}
    \label{alg:PT}
    \textbf{Initialization: } initial variables
        $z_{i}^{0}=g_{i}^{0}= q_{i}^{0}=0 \in \bR^{n} $ and fixed constant $c>0$. \\
    \textbf{For $k=0,1,\ldots$ do: }\\
    Compute the weighted averages,
        \[
        \xi_{i}^{k} = \sum_{j \in \mathcal{N}_i}  w_{ij}z_{j}^{k}, \quad
        \iota_{i}^{k}= \sum_{j \in \mathcal{N}_i}  w_{ij}g_{j}^{k}. 
        \]\\
    Updating the decision variables,
    \[
    z_{i}^{k+1} = \argmin_{z_i \in Z_i} \left\{ f_i(z_i) + (\iota_i^k-q_i^k)^Tz_i + \frac{1}{2c} \| z_i-\xi_i^k \|^2 \right\}.
    \]\\
    Update the subgradient,
        \[
        q_{i}^{k+1} = \frac{1}{c} (\xi_{i}^{k} - z_{i}^{k+1} ) + q_{i}^{k} -\iota_{i}^{k}.\] \\
    Update the tracker,
        \[
        g_{i}^{k+1} = \iota_{i}^{k}+ q_{i}^{k+1} - q_{i}^{k}.\] \\
        \textbf{end}
\end{algorithm}

%%--------------------------------------------------------------------
\subsection{Compact form of Proximal-Tracking}
In order to make the following analysis convenient, let us reformulate Algorithm \ref{alg:PT} to a compact form of the steps collectively performed by all agents running in a parallel manner. To this end, let us first denote with bold symbols the network-wide vectors obtained by stacking the corresponding quantities of all agents as
\[
    \begin{aligned}
        &\bz^{k} := {[ (z_1^k)^\top \cdots (z_N^k)^\top]}^\top,\  
          \bg^{k} := {[(g_1^k)^\top \cdots (g_N^k)^\top ]}^\top,\\
          & \bq^k := [(q_1^k)^\top \cdots (q_N^k)^\top ]^\top.
    \end{aligned}
\]
For any $\bz=[ z_1^\top \cdots z_N^\top]^\top$ with $z_i \in\bR^n$, 
define the following mapping 
\[
\bphi(\bz):= \sum_{i=1}^{N} \phi_i(z_i). 
\]
Let us also denote 
\[
\partial \bphi(\bz) := \{[(q_1)^\top \cdots (q_N)^\top ]^\top|q_i\in\partial \phi_i(z_i), i=1,\ldots,N\} .
\] 
The following lemma is easily derived from the above definitions.
\begin{lemma}\label{lem:solution}
  Under Assumptions \ref{assu:monotone} and \ref{assu:connectivity}, for any $\bz^* := \mathbb{1}_{N} \otimes z^*$, if there exists $\bq^*$ such that $\bq^*\in\partial \bphi(\bz^*)$ and  $W_{\infty}\bq^*=0$, then $z^*$ is an optimal solution to Problem \eqref{eq:IPMO}. 
\end{lemma}

From \cite{FP2022}, the sequences generated by Algorithm \ref{alg:PT} running over the whole multi-agent network can be expressed in the following compact form,
\begin{subequations} \label{eq:z_upd_gradform_vec}
    \begin{align}
        &\bz^{k+1} = W\bz^{k} - c\bg^{k+1},  \label{eq:z_upd_gradform_vec1} \\
        &\bg^{k+1} = W\bg^{k} + \bq^{k+1} - \bq^{k}, \label{eq:z_upd_gradform_vec2} \\
        &\bq^{k+1} \in \partial \bphi(\bz^{k+1}).  \label{eq:z_upd_gradform_vec3}
    \end{align}
\end{subequations}

Based on the above form, the global convergence of the Proximal-Tracking algorithm has been established in \cite{FP2022} as follows.
\begin{proposition}\label{prop:global}
Under Assumptions \ref{assu:monotone} and \ref{assu:connectivity}, the sequence $\{\bz^k\}$ generated by Algorithm \ref{alg:PT} converges to a consensus solution $\bz^* := \mathbb{1}_{N} \otimes z^*$, where $z^*$ is an optimal solution to Problem \eqref{eq:IPMO}. Moreover,  the sequence $\{\bq^k\}$ converges to a vector $\bq^*$
such that $W_{\infty}\bq^*=0$ and  $\bq^*\in\partial \bphi(\bz^*)$. 
\end{proposition}

%------------------------------------------------------------------------
%------------------------------------------------------------------------
%------------------------------------------------------------------------
\section{Convergence rate analysis}\label{sec:rate}

In order to establish the convergence rate of Algorithm \ref{alg:PT}, let us first present a series of useful auxiliary lemmas.

Denote
\[
\by^{k} := \sum_{t=0}^{k} (I-W) \bz^t.
\]
In the following lemma, we present an equivalent EXTRA-type reformulation of Proximal-Tracking, which is the starting point for linear convergence analysis.
\begin{lemma}\label{lem:yk}
Under Assumptions \ref{assu:monotone} and \ref{assu:connectivity}, for all $k\geq 0$, we have
\begin{align}
    \bz^{k+1} &= W^2 \bz^k -c\bq^{k+1}-(I-W)\by^k ,\label{eq:EXTRA_a} \\
    \by^{k+1} &= \by^{k} + (I-W)\bz^{k+1}. \label{eq:EXTRA_b}
\end{align}
\end{lemma}
\begin{proof}
The claim (\ref{eq:EXTRA_b}) is obvious from the definition of $\by^k$.
From (\ref{eq:z_upd_gradform_vec1}) and (\ref{eq:z_upd_gradform_vec2}) it follows that
\be\label{eq:aux1}
\bz^{k+1}=2W\bz^k-W^2\bz^{k-1}-c(\bq^{k+1}-\bq^k).
\ee

Let us prove (\ref{eq:EXTRA_a}) by induction. For $k=0$, the result is obvious from the fact that $\bz^0=\by^0=0$ and the subgradient update equation in line 5 of Algorithm \ref{alg:PT}. Suppose that (\ref{eq:EXTRA_a}) holds true for $k\geq 0$.
From (\ref{eq:aux1}) we have
\[
\bz^{k+2}=2W\bz^{k+1}-W^2\bz^{k}-c(\bq^{k+2}-\bq^{k+1}),
\]
which, together with (\ref{eq:EXTRA_a}) for $k$, yields
\[
\bz^{k+2}=2W\bz^{k+1}-c\bq^{k+2}-\bz^{k+1}-(I-W)\by^{k}.
\]
Noticing that $(I-W)\by^{k} = (I-W)\by^{k+1} - (I-W)^2\bz^{k+1}$ from (\ref{eq:EXTRA_b}), we further have
\[
\bz^{k+2}=W^2\bz^{k+1}-c\bq^{k+2}-(I-W)\by^{k+1},
\]
which completes the proof. \hfill $\square$
\end{proof}

Recall that $W\succ 0$ and $2I-W\succ 0$, thus $(2W-W^2)$ is also positive definite.
For all $k\geq 0$, let $\bx^{k}:=(\by^k,\bz^k)$.
For any $\bx:=(\by,\bz)\in\bR^{Nn}\times\bR^{Nn}$, define the following set-valued mapping 
\[
F(\bx):=    \begin{bmatrix}
        (I-W)\bz \\
        (2W-W^2)^{-1}(c\partial\bphi(\bz)+(I-W)\by)
    \end{bmatrix}.
\]
\begin{lemma}\label{lem:F}
Under Assumptions \ref{assu:monotone} and \ref{assu:connectivity}, for all $k\geq 0$, we have 
\be\label{eq:Fk}
 \begin{bmatrix}
        \by^{k+1} - \by^{k} \\
      (2I-W)^{-1}W\bz^k-\bz^{k+1}  
    \end{bmatrix}\in F(\bx^{k+1}).
\ee
For any $(\by^*,\bz^*)\in F^{-1}(0)$, it holds that $-\frac{1}{c}(I-W)\by^*\in\partial \bphi(\bz^*)$ and  $\bz^* := \mathbb{1}_{N} \otimes z^*$, where $z^*$ is an optimal solution to Problem \eqref{eq:IPMO}.
\end{lemma}
\begin{proof}
    From (\ref{eq:EXTRA_a}) and (\ref{eq:EXTRA_b}), we can obtain that
 \be\label{eq:aux2}
 W^2\bz^k-(2W-W^2)\bz^{k+1}=c\bq^{k+1}+(I-W)\by^{k+1}
 \ee 
 and hence
 \[
 (2I-W)^{-1}W\bz^k-\bz^{k+1}=(2W-W^2)^{-1}(c\bq^{k+1}+(I-W)\by^{k+1}),
 \]
 which, together with  (\ref{eq:z_upd_gradform_vec3}) and (\ref{eq:EXTRA_b}), implies
(\ref{eq:Fk}).

For any $(\by^*,\bz^*)\in F^{-1}(0)$, we have
\[
\begin{array}{ll}
0\in c\partial\bphi(\bz^*)+(I-W)\by^*,\\[5pt]
(I-W)\bz^*=0.
\end{array}
\]
Recall that $W= \mathcal{W} \otimes I_{n}$, the fact  $(I-W)\bz^*=0$ obviously indicates that $\bz^*$ is in the form of $\bz^* = \mathbb{1}_{N} \otimes z^*$. Let $\widehat{\by}^*:=-\frac{1}{c}(I-W)\by^*$. We have that $\widehat{\by}^*\in\partial \bphi(\bz^*)$. From Lemma \ref{lem:span}, we also have $W_{\infty}\widehat{\by}^*=0$. Therefore, from Lemma \ref{lem:solution}, it follows that $z^*$ is an optimal solution to Problem \eqref{eq:IPMO}. \hfill $\square$
\end{proof}

Based on Lemma \ref{lem:F} and partially motivated by \cite{EXTRA,PG-EXTRA},  the convergence rate of Proximal-Tracking shall be established in this section by investigating the convergence behavior of $(\bz^k,\by^k)$ to $(\bz^*,\by^*)\in F^{-1}(0)$.

 For any $\bx=(\by,\bz)\in\bR^{Nn}\times\bR^{Nn}$, define
\[
\|\bz\|_{(2W-W^2)}^2:=\bz^T(2W-W^2)\bz
\]
and
\[
\|\bx\|_A^2:=\|\by\|^2+\|\bz\|_{(2W-W^2)}^2.
\]
Moreover, for a set $\Omega$, denote 
\[\dist_A^2(\bx,\Omega)=\min_{\bx'\in\Omega}\|\bx-\bx'\|_A^2.\]
In the following lemma we obtain the recursive relation of $\|\bx^k-\bx^*\|_A^2$.
\begin{lemma}\label{lem:recur}
    Under Assumptions \ref{assu:monotone} and \ref{assu:connectivity}, for any $\bx^*=(\by^*,\bz^*)\in F^{-1}(0)$ and for all $k\geq 0$, one has
    \[
    \|\bx^k-\bx^*\|_A^2\geq \|\bx^{k+1}-\bx^*\|_A^2+\dist_A^2(0,F(\bx^{k+1})).
    \]
\end{lemma}
\begin{proof}
Since $(\by^*,\bz^*)\in F^{-1}(0)$, we have
\[
-\frac{1}{c}(I-W)\by^*\in\partial \bphi(\bz^*),\quad (I-W)\bz^*=0.
\]
Then, one has $\bz^*=W\bz^*$ and hence
\[
(2I-W)^{-1}W\bz^*=\bz^*.
\]
Note that $\lambda_{\max}(W)=1$ and hence $\lambda_{\max}((2I-W)^{-1}W)=\lambda_{\max}(W)/(2-\lambda_{\max}(W))=1$, then
\be\label{eq:aux3}
\begin{array}{ll}
\|\bz^k-\bz^*\|_{(2W-W^2)}^2\\[5pt]
\geq \|(2I-W)^{-1}W(\bz^k-\bz^*)\|_{(2W-W^2)}^2\\[5pt]
= \|(2I-W)^{-1}W\bz^k-\bz^*\|_{(2W-W^2)}^2\\[5pt]
=\|\bz^{k+1}-\bz^*\|_{(2W-W^2)}^2+\|(2I-W)^{-1}W\bz^k-\bz^{k+1}\|_{(2W-W^2)}^2\\[5pt]
\quad+2\iprod{\bz^{k+1}-\bz^*}{W^2\bz^k-(2W-W^2)\bz^{k+1}}\\[5pt]
=\|\bz^{k+1}-\bz^*\|_{(2W-W^2)}^2+\|(2I-W)^{-1}W\bz^k-\bz^{k+1}\|_{(2W-W^2)}^2\\[5pt]
\quad+2\iprod{\bz^{k+1}-\bz^*}{c\bq^{k+1}+(I-W)\by^{k+1}},
\end{array}
\ee
where the last equality is from (\ref{eq:aux2}).
From (\ref{eq:EXTRA_b}) and $(I-W)\bz^*=0$ it follows that $\by^{k+1}- \by^{k} = (I-W)(\bz^{k+1}-\bz^*)$ and thus
\be\label{eq:aux4}
\begin{array}{ll}
\|\by^k-\by^*\|^2&=\|\by^{k+1}-\by^*\|^2+\|\by^k-\by^{k+1}\|^2\\[5pt]
 &\quad+2\iprod{\by^{k+1}- \by^{k} }{\by^*-\by^{k+1}}\\[5pt]
&=\|\by^{k+1}-\by^*\|^2+\|\by^k-\by^{k+1}\|^2\\[5pt]
 &\quad+2\iprod{\bz^{k+1}-\bz^*}{(I-W)(\by^*-\by^{k+1})}.
\end{array}
\ee
Noticing that $-\frac{1}{c}(I-W)\by^*\in\partial \bphi(\bz^*)$  and (\ref{eq:z_upd_gradform_vec3}), 
we have from the monotonicity of $\partial\bphi$ that
\be\label{eq:aux5}
\begin{array}{ll}
2\iprod{\bz^{k+1}-\bz^*}{c\bq^{k+1}+(I-W)\by^{k+1}+(I-W)(\by^*-\by^{k+1})}\\[5pt]
=2c\iprod{\bz^{k+1}-\bz^*}{\bq^{k+1}+\frac{1}{c}(I-W)\by^*}\\[5pt]
\geq 0.
\end{array}
\ee
By combining (\ref{eq:aux3}), (\ref{eq:aux4}) and (\ref{eq:aux5}) together and rearranging terms, we derive 
    \[
    \begin{array}{ll}
    \|\bx^k-\bx^*\|_A^2&\geq \|\bx^{k+1}-\bx^*\|_A^2+\|\by^k-\by^{k+1}\|^2\\[5pt]
    &\quad +\|(2I-W)^{-1}W\bz^k-\bz^{k+1}\|_{(2W-W^2)}^2.
    \end{array}
    \]
Finally, from (\ref{eq:Fk}), we have
\[
  \begin{array}{ll}
\|\by^k-\by^{k+1}\|^2+\|(2I-W)^{-1}W\bz^k-\bz^{k+1}\|_{(2W-W^2)}^2\\[5pt]\geq \dist_A^2(0,F(\bx^{k+1})).
    \end{array}
\]
The proof is completed.\hfill $\square$
\end{proof}

In the sequel, we require additional assumptions to establish the linear convergence of Algorithm \ref{alg:PT}. 
\begin{assumption}\label{assu:strong}
Assume that each local objective function $f_i$ of agent $i$ is strongly convex with modulus $\mu>0$.
\end{assumption}
Assumption \ref{assu:strong}
 is a standard assumption to study the linear convergence of distributed algorithms, see \cite{EXTRA,DIGing,QL2018} for example.
 Under Assumptions \ref{assu:monotone} and \ref{assu:strong}, we have that $\bphi$ is proper, lower semicontinuous, and $\mu$-strongly convex, then from \cite[Theorem 12.17]{RW1998} and \cite[Exercise 12.59]{RW1998} it is well known that the mapping $\partial\bphi$ is maximal monotone and strongly monotone with modulus $\mu$. 
 %Furthermore, from \cite[Proposition 12.54]{RW1998} and \cite[Proposition 3G.1]{DR2009}, one can obtain that $\partial\bphi$ is metric regular. 

Under Assumption \ref{assu:strong}, we have that the solution to Problem (\ref{eq:IPMO}) is unique. For convenience, throughout we let $z^*$ be the unique solution and $\bz^* := \mathbb{1}_{N} \otimes z^*$. Recall from Proposition \ref{prop:global} and Lemma \ref{lem:span} that $\cQ^*:=\partial\bphi(\bz^*)\cap\Range(I-W)\neq\emptyset$.
\begin{assumption}\label{assu:calm}
There exist $\ell\geq 0$, $\varrho\geq 0$ and $\delta\in(0,+\infty]$ such that for any $\bz$ with $\partial\bphi(\bz)\neq\emptyset$ and $\|\bz-\bz^*\|\leq\delta$ and every $\bq\in\partial\bphi(\bz)$,
\be\label{eq:calm}
\dist(\bq,\cQ^*)\leq\ell\|\bz-\bz^*\|+\varrho\dist(\bq,\Range(I-W)).
\ee
\end{assumption}
Assumption \ref{assu:calm} can be seen as a combination of a calmness property of $\partial\bphi$ at $\bz^*$ and a linear regularity property of the pair $(\partial\bphi(\bz^*),\Range(I-W))$.  If $\delta=+\infty$, we say that (\ref{eq:calm}) holds globally. Sufficient conditions to ensure Assumption \ref{assu:calm}  are collected in Remark \ref{rem:calm}. Assumption \ref{assu:calm} will be used to derive the metric subregularity of $F$ at $0$, which is a key property to establish the linear convergence of Proximal-Tracking.

\begin{lemma}\label{lem:metric}
Suppose that Assumptions \ref{assu:monotone}, \ref{assu:connectivity}, \ref{assu:strong} and \ref{assu:calm} hold true.
Set
\begin{equation}\label{eq:kappa}
C:=\frac{1}{c\mu}+\frac{\ell}{\mu\sigma}+\sqrt{\frac{1+\varrho}{2\sigma c\mu}},\ \kappa:=\sqrt{C^2+\Bigl(\frac{1+\varrho+c\ell C}{\sigma}\Bigr)^{2}} .
\end{equation}
Then for every $\bx=(\by,\bz)$ with $\partial\bphi(\bz)\neq\emptyset$ and
$\|\bz-\bz^*\|\le\delta$,
\begin{equation}\label{eq:new}
\dist\bigl(\bx,F^{-1}(0)\bigr)\ \le\ \kappa\,\dist\bigl(0,F(\bx)\bigr).
\end{equation}
In particular, if (\ref{eq:calm}) holds globally, then \eqref{eq:new} holds for
every $\bx$ with $\partial\bphi(\bz)\neq\emptyset$.
\end{lemma}
\begin{proof}
For notational brevity, write
\[
P:=I-W,\ M:=2W-W^2,\ S:=I-W_{\infty}.
\]
The set $\partial\bphi(\bz)$ is nonempty, closed
and convex, hence so is
$F(\bx)=\{P\bz\}\times M^{-1}\bigl(c\,\partial\bphi(\bz)+P\by\bigr)$, and the
distance $\eta:=\dist(0,F(\bx))$ is attained at some
$\bu=(\bu_1,\bu_2)\in F(\bx)$. Then, there is $\bq\in\partial\bphi(\bz)$ such that
\begin{equation}\label{eq:pf1}
\bu_1=P\bz,\qquad M\bu_2=c\,\bq+P\by .
\end{equation}

\emph{Step 1: construction of a candidate solution.}
Let $\bq^*\in\cQ^*$ satisfy $\|\bq-\bq^*\|=\dist(\bq,\cQ^*)$ and let
\[
\by^*:=W_\infty\by-c\,P^\dagger\bq^*,\qquad \bx^*:=(\by^*,\bz^*).
\]
Since $PW_\infty=0$ and $PP^\dagger$ is the identity on $\Range(P)\ni\bq^*$, we get $P\by^*=-c\,\bq^*$, whence
$0\in M^{-1}(c\,\partial\bphi(\bz^*)+P\by^*)$; together with $P\bz^*=0$ this
gives $\bx^*\in F^{-1}(0)$. 
Put \[a:=\|\bz-\bz^*\|,\quad b:=\|\by-\by^*\|.\]
Note that $\by-\by^*=S\by+cP^\dagger\bq^*\in\Range(P)$.

\emph{Step 2: the residual controls the consensus part of $\bq$.}
Applying $W_\infty$ to \eqref{eq:pf1} and using $W_\infty P=0$,
$W_\infty M=MW_\infty$ and $M\bv=\bv$ on $\Range(W_\infty)$,
\[
c\,W_\infty\bq=W_\infty M\bu_2=W_\infty\bu_2 ,
\]
and hence, 
\begin{equation}\label{eq:pf2}
\dist\bigl(\bq,\Range(P)\bigr)=\|W_\infty\bq\|=\tfrac1c\|W_\infty\bu_2\|\le\tfrac1c\|\bu_2\| .
\end{equation}
Combining \eqref{eq:pf2} with \eqref{eq:calm} yields
\begin{equation}\label{eq:pf3}
c\,\|\bq-\bq^*\|\ \le\ c\ell\,a+\varrho\,\|\bu_2\| .
\end{equation}

\emph{Step 3: bounding $b$.}
From \eqref{eq:pf1} and $P\by^*=-c\bq^*$ it follows
$P(\by-\by^*)=M\bu_2-c(\bq-\bq^*)$. Since $\by-\by^*\in\Range(P)$, thus $\|P(\by-\by^*)\|\geq\sigma\|\by-\by^*\|$, $\lambda_{\max}(M)=1$ and
 \eqref{eq:pf3} give
\begin{equation}\label{eq:pf4}
b\ \le\ \frac{1}{\sigma}\bigl\|M\bu_2-c(\bq-\bq^*)\bigr\|
\ \le\ \frac{1}{\sigma}\Bigl((1+\varrho)\|\bu_2\|+c\ell\,a\Bigr).
\end{equation}

\emph{Step 4: bounding $a$.}
Since $\bq\in\partial\bphi(\bz)$, $\bq^*\in\partial\bphi(\bz^*)$ and
$\partial\bphi$ is strongly monotone with modulus $\mu$,
\[
\begin{array}{ll}
c\mu\,a^2\ \le\ c\,\iprod{\bq-\bq^*}{\bz-\bz^*}\\
=\iprod{M\bu_2-P(\by-\by^*)}{\bz-\bz^*}\\
=\iprod{M\bu_2}{\bz-\bz^*}-\iprod{\by-\by^*}{\bu_1},
\end{array}
\]
where we used $P(\bz-\bz^*)=P\bz=\bu_1$. Hence, 
\begin{equation}\label{eq:pf5}
c\mu\,a^2\ \le\ \|\bu_2\|\,a+b\,\|\bu_1\| .
\end{equation}
Substituting \eqref{eq:pf4} into \eqref{eq:pf5} yields $c\mu a^2\le\alpha a+\beta$ with
\[
\alpha:=\|\bu_2\|+\frac{c\ell}{\sigma}\|\bu_1\|,\qquad
\beta:=\frac{1+\varrho}{\sigma}\|\bu_1\|\,\|\bu_2\| .
\]
Solving the quadratic inequality yields
\[
a\ \le\ \frac{\alpha+\sqrt{\alpha^2+4c\mu\beta}}{2c\mu}\ \le\ \frac{\alpha}{c\mu}+\sqrt{\frac{\beta}{c\mu}} .
\]
Because $\|\bu_1\|\le\eta$, $\|\bu_2\|\le\eta$ and
$\|\bu_1\|\|\bu_2\|\le\frac12(\|\bu_1\|^2+\|\bu_2\|^2)=\frac12\eta^2$, we obtain
$\alpha\le(1+\frac{c\ell}{\sigma})\eta$ and $\beta\le\frac{1+\varrho}{2\sigma}\eta^2$, hence
\begin{equation}\label{eq:pf6}
a\ \le\ C\,\eta
\end{equation}
with $C$ as in \eqref{eq:kappa}.

Finally,
inserting \eqref{eq:pf6} into \eqref{eq:pf4} indicates
$b\le\frac{1+\varrho+c\ell C}{\sigma}\eta$, and therefore
\[
\dist(\bx,F^{-1}(0))\ \le\ \|\bx-\bx^*\|=\sqrt{a^2+b^2}\ \le\ \kappa\,\eta.
\]
The proof is completed.
 \hfill $\square$
\end{proof}

We now present the main theorem that shows that the sequence $\dist_A^2(\bx^{k},F^{-1}(0))$ converges to 0 $Q$-linearly.
\begin{theorem}\label{th:rate}
Suppose that Assumptions \ref{assu:monotone}, \ref{assu:connectivity}, \ref{assu:strong} and \ref{assu:calm} hold true. Then, if (\ref{eq:calm}) holds globally, for all $k\geq 0$, 
\[
\dist_A^2(\bx^{k+1},F^{-1}(0))\leq \frac{1}{1+\theta}\dist_A^2(\bx^{k},F^{-1}(0)),
\]
where $\theta:=(2\underline{\lambda}-\underline{\lambda}^2)/\kappa^2$, $\underline{\lambda}=\lambda_{\min}(W)>0$ and $\kappa$ is given by (\ref{eq:kappa}).
If instead (\ref{eq:calm}) holds for $\delta<\infty$, then the same estimate holds for all $k\geq k_0$, where $k_0$ is an index with $\|\bz^{k}-\bz^*\|\leq\delta$ for all $k\geq k_0$ (Such an index $k_0$ exists from Proposition \ref{prop:global}).
\end{theorem}
\begin{proof}
    We only consider the case $\delta=+\infty$.
From Lemma \ref{lem:recur}, we have that 
\[
\|\bx^k-\bx^*\|_A^2\geq \dist_A^2(\bx^{k+1},F^{-1}(0))+\dist_A^2(0,F(\bx^{k+1})).
\]
Noticing that this inequality holds for all $\bx^*\in F^{-1}(0)$, we can conclude that   
\be\label{eq:aux6}
\dist_A^2(\bx^{k},F^{-1}(0))\geq \dist_A^2(\bx^{k+1},F^{-1}(0))+\dist_A^2(0,F(\bx^{k+1})).
\ee

From Lemma \ref{lem:metric}, one has
\be\label{eq:aux7}
\dist^2(\bx^{k+1},F^{-1}(0))\leq \kappa^2\dist^2(0,F(\bx^{k+1})).
\ee
To proceed, it is necessary to firstly establish the relation between $\dist_A^2(\cdot,\cdot)$ and $\dist^2(\cdot,\cdot)$. Suppose that $\bu=(\bq,\bv)\in F(\bx^{k+1})$ such that $\dist_A^2(0,F(\bx^{k+1}))=\|\bu\|_A^2$,  then
\[
\begin{array}{ll}
\dist_A^2(0,F(\bx^{k+1}))&=\|\bq\|^2+\|\bv\|_{(2W-W^2)}^2\\[5pt]
&\geq (2\underline{\lambda}-\underline{\lambda}^2)(\|\bq\|^2+\|\bv\|^2)\\[5pt]
&\geq (2\underline{\lambda}-\underline{\lambda}^2)\dist^2(0,F(\bx^{k+1})),
\end{array}
\]
where the first inequality follows from $\lambda_{\min}(2W-W^2)=2\underline{\lambda}-\underline{\lambda}^2$ and $2\underline{\lambda}-\underline{\lambda}^2\leq 1$. Similarly, noticing that $\lambda_{\max}(2W-W^2)=1$, we obtain
\[
\dist^2(\bx^{k+1},F^{-1}(0))\geq\dist_A^2(\bx^{k+1},F^{-1}(0)).
\]
Therefore, it follows from (\ref{eq:aux7}) that
\[
\dist_A^2(0,F(\bx^{k+1}))\geq \theta\dist_A^2(\bx^{k+1},F^{-1}(0)),
\]
where $\theta:=(2\underline{\lambda}-\underline{\lambda}^2)/\kappa^2$. This inequality, together with (\ref{eq:aux6}), implies that
\[
\dist_A^2(\bx^{k},F^{-1}(0))\geq (1+\theta)\dist_A^2(\bx^{k+1},F^{-1}(0)).
\]
The proof is completed. \hfill $\square$
\end{proof}

From Theorem \ref{th:rate}, we can also easily derive that each sequence $(z_i^k)$ of agent $i$ converges to the unique solution $z^*$ $R$-linearly.
\begin{corollary}\label{coro:rate}
Under the conditions of Theorem \ref{th:rate} with $\delta=+\infty$, for all $k\geq 0$, 
\[
\|\bz^k-\bz^*\|\leq \frac{1}{\sqrt{(2\underline{\lambda}-\underline{\lambda}^2)}(1+\theta)^{k/2}}\dist_A(0,F^{-1}(0)).
\]
\end{corollary}
\begin{proof}
From Theorem \ref{th:rate}, we have
\[
\dist_A^2(\bx^{k},F^{-1}(0))\leq\frac{1}{(1+\theta)^k}\dist_A^2(\bx^{0},F^{-1}(0)).
\]
Recall that $\bz^0=0$ and $\by^0=0$, and hence $\bx^0=0$. Since the solution $z^*$ of problem \eqref{eq:IPMO} is unique, from Lemma \ref{lem:F} we have that the coordinate projection of $F^{-1}(0)$ onto $\bz$ is a singleton. Therefore,
\[
\dist_A^2(\bx^{k},F^{-1}(0))\geq \|\bz^k-\bz^*\|_{2W-W^2}^2\geq (2\underline{\lambda}-\underline{\lambda}^2)\|\bz^k-\bz^*\|^2.
\]
The proof is completed.\hfill $\square$
\end{proof}

We conclude this section with a series of remarks.
\begin{remark}\label{rem:calm}
Let us introduce a few concrete examples such that Assumption \ref{assu:calm} is satisfied.
\begin{enumerate}
\item[(i)] \emph{Lipschitz smooth plus piecewise polyhedral case.}
Suppose that $f_i=s_i+r_i$ in Problem (\ref{eq:IPMO}), where $s_i$ is continuously differentiable and $\nabla s_i$ is locally Lipschitz around
$z^*$, and the subdifferential sum rule
\[
\partial\phi_i(z)=\nabla s_i(z)+R_i(z),
\qquad R_i:=\partial r_i+N_{Z_i},
\]
holds locally. Assume in addition that each $R_i$ is piecewise polyhedral and
that each $R_i(z^*)$ is polyhedral. Both conditions hold, under the usual
sum-rule qualification, when $\gph(\partial r_i)$ is piecewise polyhedral and
$Z_i$ is a convex polyhedron: piecewise polyhedrality is preserved under sums,
and a proper lower semicontinuous convex function whose subdifferential is
piecewise polyhedral is piecewise linear--quadratic
\cite[Proposition~12.30]{RW1998}, so that
$\partial r_i(z^*)$ is a polyhedral convex set, while $N_{Z_i}(z^*)$ is a
polyhedral convex cone.
Robinson's theorem \cite{Robinson1981} implies that $R_i$ is
upper Lipschitz at $z^*$. Consequently $\partial\phi_i=\nabla s_i+R_i$, and
hence $\partial\bphi$, is upper Lipschitz at the solution. Thus there are
$\ell_0\ge0$ and $\delta>0$ such that
\[
\dist\bigl(\bq,\partial\bphi(\bz^*)\bigr)
\le \ell_0\|\bz-\bz^*\|
\]
for every $\bq\in\partial\bphi(\bz)$
and
$
\|\bz-\bz^*\|\le\delta.
$
Moreover $\partial\phi_i(z^*)=\nabla s_i(z^*)+R_i(z^*)$ is polyhedral, hence so
is the product $\partial\bphi(\bz^*)$; together with the subspace $\Range(P)$
it forms a pair of polyhedral convex sets whose intersection $\cQ^*$ is
nonempty. Hoffman's error bound \cite{Hoffman} therefore gives
$\varrho_0\ge0$ such that
\[
\dist(\bv,\cQ^*)
\le\varrho_0\bigl(
\dist(\bv,\partial\bphi(\bz^*))+\dist(\bv,\Range(P))\bigr)
\]
for all $\bv$.
Applying this estimate at $\bv=\bq$ yields \eqref{eq:calm} with
$\ell=\varrho_0\ell_0$ and $\varrho=\varrho_0$. 
\item[(ii)] \emph{Smooth case.} If $Z_i=\bR^n$ and each $\nabla f_i$ is
Lipschitz with constant $\ell_i$, then $\partial\bphi=\nabla\bphi$ is
single-valued and $\cQ^*=\{\nabla\bphi(\bz^*)\}$, so \eqref{eq:calm} holds
globally with $\ell=\max_i\ell_i$ and $\varrho=0$.
\end{enumerate}
\end{remark}

\begin{remark}\label{rem:assumptions}
 One of the original advantages of Proximal-Tracking is that it does not require differentiability, smoothness or Lipschitz continuity of the local cost function. Let us point out that the analysis in this section does not conflict with this advantage. Assumption \ref{assu:calm} keeps the nonsmooth setting genuinely nonsmooth. It covers piecewise polyhedral regularizers such as $\ell_1$, elastic net, anisotropic TV and polyhedral indicator functions without requiring smoothness of the full objective.
\end{remark}

\begin{remark}\label{rem:sublinear}
Let us use a toy example to show that, besides strong convexity assumption, additional stability assumptions such as Assumption \ref{assu:calm} are necessary to derive the linear convergence of Proximal-Tracking. Consider a two-agent network that satisfies Assumption \ref{assu:connectivity},
\[
N=2,\ n=1,\ \mathcal{W}=\begin{bmatrix}1-w & w\\ w& 1-w\end{bmatrix},\ w\in(0,1/2).
\]
Let $Z_1=Z_2=\bR$ and define
\[
g(z):=\frac{\mu}{2}z^2+\frac{2}{3}|z|^{3/2},
\]
where $\mu>0$ is a constant. The two local cost functions are
\[
f_1(z):=g(z)+dz,\quad f_2(z):=g(z)-dz,
\]
where $d>0$ is a constant. Each $f_i$ is finite, closed and $\mu$-strongly convex, and thus Assumptions \ref{assu:monotone} and \ref{assu:strong} hold. The unique solution is $z^*=0$. By direct calculations, we have
\[
\|\bz^k-\bz^*\|\sim\frac{\sqrt{2}c^2}{(2w)^4}\frac{1}{k^2},
\]
which implies that Algorithm \ref{alg:PT} converges with a sublinear rate.
\end{remark}
%------------------------------------------------------------------------
%------------------------------------------------------------------------
%------------------------------------------------------------------------
\section{Augmented Lagrangian Tracking} \label{sec:alt}
Based on Proximal-Tracking, an algorithm named as Augmented Lagrangian Tracking (ALT) is proposed in \cite{FP2023} for solving distributed optimization with equality and inequality coupling constraints. Under mild conditions, the global convergence of ALT is established in \cite{FP2023}. However, the convergence rate of ALT has not been studied yet. In this section, we show that the linear convergence rate of ALT could be established by following a similar analysis as in Section \ref{sec:rate}.

Consider the following constraint-coupled optimization problem,
\be\label{eq:CP}
    \begin{aligned}
        \min_{x_i} &\; \sum_{i=1}^{N} f_i(x_i) ,\\
        \text{s.t. }& \sum_{i=1}^N(A_ix_i-b_i)=0,\ \sum_{i=1}^Nh_i(x_i)\leq 0,\\
        &x_i\in X_i,\ i=1,\ldots,N,
    \end{aligned}
\ee
where $f_i:\bR^{n_i}\rightarrow\bR$ is the local cost function of agent $i$, $A_i\in\bR^{p\times n_i}$, $h_i:\bR^{n_i}\rightarrow\bR^q$ and $X_i\subseteq\bR^{n_i}$. Let us point out that it is not required that the local set $X_i$ be compact, nor that $f_i$ and $h_i$ be smooth.

Let $\bx=[ x_1^\top \cdots x_N^\top]^\top$ with $x_i \in\bR^{n_i}$.
The Lagrangian function of Problem (\ref{eq:CP}) is given by
\[
L(\bx,\lambda,\mu):=\sum_{i=1}^N[f_i(x_i)+\lambda^T(A_ix_i-b_i)+\mu^Th_i(x_i)].
\]
The dual problem of Problem (\ref{eq:CP}) is in the form of
\be\label{eq:DCP}
\max\limits_{\lambda\in\bR^p,\mu\in\bR^q_+}\sum_{i=1}^N\varphi_i(\lambda,\mu),
\ee
where
\be\label{eq:dual}
\varphi_i(\lambda,\mu)=\inf_{x_i\in X_i}f_i(x_i)+\lambda^T(A_ix_i-b_i)+\mu^Th_i(x_i).
\ee
Under standard convexity and strong duality assumptions, it is known that solving Problem (\ref{eq:CP}) is equivalent to solving its dual problem (\ref{eq:DCP}). Observing that Problem (\ref{eq:DCP}) is in the same formulation of Problem (\ref{eq:IPMO_eqv}), we have that Proximal-Tracking can be easily adapted to solve Problem (\ref{eq:DCP}). Accordingly, as shown in \cite{FP2023}, the steps of ALT for agent $i$ are as follows,
\[
\begin{array}{ll}
\displaystyle\begin{bmatrix} l_i^k\\ m_i^k\end{bmatrix}=\sum_{j\in\cN_i}w_{ij}\begin{bmatrix} \lambda_j^k\\ \mu_j^k\end{bmatrix},\ \begin{bmatrix} \delta_i^k\\ \gamma_i^k\end{bmatrix}=\sum_{j\in\cN_i}w_{ij}\begin{bmatrix} d_j^k\\ g_j^k\end{bmatrix},\\[10pt]
\begin{bmatrix} \lambda_i^{k+1}\\ \mu_i^{k+1}\end{bmatrix}=\argmin\limits_{\lambda\in\bR^p,\mu\in\bR^q_+}-\varphi_i(\lambda,\mu)+\begin{bmatrix} \delta_i^k-u_i^k\\ \gamma_i^k-v_i^k\end{bmatrix}^T\begin{bmatrix} \lambda\\ \mu\end{bmatrix}\\[10pt]\quad\quad\quad\quad\quad\quad\quad\quad\quad\quad+\frac{1}{2c}\left\|\begin{bmatrix} \lambda-l_i^k\\ \mu-m_i^k\end{bmatrix}\right\|^2,\\[10pt]
\begin{bmatrix} u_i^{k+1}\\ v_i^{k+1}\end{bmatrix}=-\frac{1}{c}\begin{bmatrix} \lambda_i^{k+1}-l_i^k\\ \mu_i^{k+1}-m_i^k\end{bmatrix}-\begin{bmatrix} \delta_i^k-u_i^k\\ \gamma_i^k-v_i^k\end{bmatrix},\\[10pt]
\begin{bmatrix} d_i^{k+1}\\ g_i^{k+1}\end{bmatrix}=\begin{bmatrix} \delta_i^k\\ \gamma_i^k\end{bmatrix}+\begin{bmatrix} u_i^{k+1}\\ v_i^{k+1}\end{bmatrix}-\begin{bmatrix} u_i^{k}\\ v_i^{k}\end{bmatrix},
\end{array}
\]
where $(\lambda_i^k,\mu_i^k)$ are local estimates of the optimal solution to Problem (\ref{eq:DCP}). As stated previously, the global convergence has been well established in \cite{FP2023}. 

Since ALT is actually Proximal-Tracking adapted to Problem (\ref{eq:DCP}), naturally, the linear convergence rate of ALT can be derived very similarly  from the convergence rate analysis of Proximal-Tracking as in Section \ref{sec:rate}. However, as demonstrated previously, a key assumption to derive the linear rate of ALT is Assumption \ref{assu:calm}. In \cite{GZ2025}, an example is provided to ensure that each $\partial(-\varphi_i)$ is metric regular: each local objective function $f_i(x)=s_i(x)+r_i(x)$, where $s_i$ is smooth and strongly convex, $\partial r_i$ is polyhedral, i.e., $\gph(\partial r_i)$ is a union of finite polyhedral convex sets, $h_i$ is affine and $X_i$ is a convex polyhedron.
By following the discussion of item (i) in Remark \ref{rem:calm}, we can conclude that this example does satisfy Assumption \ref{assu:calm}.
Therefore, in this setting, the linear convergence rate of ALT can be established. We omit the detailed proof here for saving space.
%------------------------------------------------------------------------
%------------------------------------------------------------------------
%------------------------------------------------------------------------


\printcredits

\section*{Data availability}
No data was used for the research described in the article.

\section*{Acknowledgement}
This work is supported by the National Natural Science Foundation of China (12271076).
\bibliographystyle{cas-model2-names}

\bibliography{ref}



\end{document}